\documentclass[reqno,11pt]{amsart}
\usepackage{amsmath}
\usepackage{amssymb}
\usepackage{tabularx}
\usepackage{enumerate}
\usepackage[dvips]{psfrag,graphicx}

\usepackage{mathrsfs}

\usepackage{amsfonts,amssymb,amsmath}
\usepackage{epsfig}

\makeatletter
\@addtoreset{equation}{section}
  
\makeatother
\newtheorem{thm}{Theorem}[section]
\newtheorem{lem}[thm]{Lemma}

\newtheorem{rem}[thm]{Remark}

\begin{document}
\title[Full cross-diffusion limit in the stationary SKT model]
{Full cross-diffusion limit 
in the stationary Shigesada-Kawasaki-Teramoto model}
 \thanks{This research was
partially supported by JSPS KAKENHI Grand Number 19K03581.}
\author[K. Kuto]{Kousuke Kuto$^\dag$}
\thanks{$\dag$ Department of Applied Mathematics, 
Waseda University, 
3-4-1 Ohkubo, Shinjuku-ku, Tokyo 169-8555, Japan.}
\thanks{{\bf E-mail:} \texttt{kuto@waseda.jp}}
\date{\today}

\begin{abstract} 
This paper studies the asymptotic behavior of  
coexistence steady states of the Shigesada-Kawasaki-Teramoto model 
as both cross-diffusion coefficients tend to infinity at the same rate.
In the case when either one of two cross-diffusion coefficients
tends to infinity,
Lou and Ni \cite{LN2} derived a couple of limiting systems, 
which characterize the asymptotic behavior of 
coexistence steady states. 
Recently, a formal observation by
Kan-on \cite{Ka3} implied the existence of a limiting system 
including the nonstationary problem as both
cross-diffusion coefficients tend to infinity at the same rate.
This paper gives a rigorous proof of 
his observation as far as the stationary problem.
As a key ingredient of the proof,
we establish a uniform $L^{\infty}$ estimate
for all steady states.
Thanks to this a priori estimate,
we show that the asymptotic profile of coexistence steady states
can be characterized by a solution of either of two limiting systems.
\end{abstract}

 \maketitle

\section{Introduction}
This paper is concerned with the following Neumann problem of
nonlinear elliptic equations:
\begin{equation}\label{SKT}
\begin{cases}
\Delta [\,(d_{1}+\alpha v)u\,]+
f(u,v)=0
\ \ &\mbox{in}\ \Omega,\\
\Delta [\,(d_{2}+\beta u)v\,]+
g(u,v)=0
\ \ &\mbox{in}\ \Omega,\\
u\ge 0,\ \ v\ge 0
\ \ &\mbox{in}\ \Omega,\\
\partial_{\nu}u=\partial_{\nu}v=0
\ \ &\mbox{on}\ \partial\Omega,
\end{cases}
\end{equation}
where
\begin{equation}\label{fgdef}
f(u,v):=u(a_{1}-b_{1}u-c_{1}v),\qquad
g(u,v):=v(a_{2}-b_{2}u-c_{2}v).
\end{equation}
Here $\Omega$
is a bounded domain in $\mathbb{R}^{N}$ 
with smooth boundary $\partial\Omega $;
$\Delta :=\sum^{N}_{j=1}\partial^{2}/\partial x_{j}^{2}$
is the usual Laplace operator;
$\nu (x)$
is the outer unit normal vector
at $x\in\partial\Omega $, and
$
\partial_{\nu }u=
\nu (x)\cdot\nabla u
$
represents the out-flux of $u$;
coefficients
$a_{i}$,
$b_{i}$,
$c_{i}$ and
$d_{i}$
$(i=1,2)$
are positive constants;
$\alpha$
and
$\beta$
are nonnegative constants.
System \eqref{SKT}
is the stationary problem
of a Lotka-Volterra competition model in which 
unknown functions $u(x)$ and $v(x)$ represent the stationary population densities
of two competing species in the habitat $\Omega$.
In the reaction terms,
$a_{i}$ represent the birth rates of the respective species,
$b_{1}$ and $c_{2}$ denote the intra-specific competition coefficients, and
$c_{1}$ and $b_{2}$ denote the inter-specific competition coefficients.
In the diffusion terms,
$d_{1}\Delta u$ and $d_{2}\Delta v$ represent 
the linear diffusion determined by
the dispersive force associated 
with random movement of each species,
whereas
$\alpha\Delta (uv)$ and $\beta\Delta (uv)$ denote
the nonlinear diffusion caused by  
the population pressure resulting from interference
between different species.
The interaction term $\Delta (uv)$ is often referred to as the
{\it cross-diffusion}.
See a book by Okubo and Levin \cite{OL} for 
modellings of the biological diffusion.
Such a Lotka-Volterra competition system with cross-diffusion 
(and additional) terms was proposed by Shigesada, Kawasaki and Teramoto
\cite{SKT}.
Beyond their bio-mathematical aim to realize 
{\it segregation phenomena} of two competing species observed in 
ecosystems, a lot of pure mathematicians have studied a class of  
Lotka-Volterra systems with cross-diffusion as a prototype of 
diffusive interactions.
Today, such a class of Lotka-Volterra system with cross-diffusion 
is referred as {\it the SKT model} celebrating the authors of \cite{SKT}.
See e.g., the book chapters by J\"ungel \cite{Ju}, Ni \cite{Ni},
and Yamada \cite{Yam1, Yam2}
as surveys for mathematical works relating to the SKT model.

Since there are a lot of papers 
studying the stationary SKT model like \eqref{SKT}.
we just give a brief history of studies.
Immediately after the proposal by \cite{SKT},
the group of Mimura began to study \eqref{SKT}.
Their main methods in 1980s are 
the bifurcation (\cite{MK}) and
the singular perturbation (\cite{Mi, MNTT}), and moreover,
Kan-on \cite{Ka1} identified some criteria for
ensuring the stability/instability of
nonconstant solutions obtained by \cite{MNTT}.
After the middle of 1990s,
a couple of papers by Lou and Ni
(\cite{LN1, LN2}) advanced the understanding of \eqref{SKT} considerably.
By the combination of the energy method, the singular perturbation
and the degree theory,
the first paper
\cite{LN1} obtained some a priori estimates of solutions
and sufficient conditions for the existence/nonexistence of 
nonconstant solutions of \eqref{SKT} with some additional terms.
In the sequel \cite{LN2}, they studied the asymptotic behavior of
nonconstant solutions as $\alpha\to\infty$ (with fixed small $\beta\ge 0$).
For such a {\it cross-diffusion limit} procedure for \eqref{SKT},
there is a difficulty to derive the a priori $L^{\infty}$ estimate
of all solutions independently of $\alpha\ge 0$.
They established the a priori estimate \cite[Theorem 2.3]{LN2}
in case when $N\le 3$ and $\beta\ge 0$ is sufficiently small, and moreover,
found the following a couple of {\it limiting systems} (or shadow systems)
which characterize the asymptotic behavior of nonconstant solutions
as $\alpha\to\infty$:
\begin{thm}[\cite{LN2}]\label{LNthm}
Suppose that $N\le 3$,
$a_{1}/a_{2}\neq b_{1}/b_{2}$,
$a_{1}/a_{2}\neq c_{1}/c_{2}$
and 
$a_{2}/d_{2}$ is not equal to any eigenvalue of $-\Delta$
with homogeneous Neumann boundary condition on $\partial\Omega$.
Let $\{(u_{n}, v_{n})\}$ be any sequence of
positive nonconstant solutions of
\eqref{SKT} with $\alpha=\alpha_{n}\to\infty$.
Then there exists a small 
$\delta =\delta (a_{i}, b_{i}, c_{i}, d_{i})>0$ such that
if $\beta\le\delta$,
either of 
the following (i) or (ii) occurs;
\begin{enumerate}[(i)]
\item
there are a positive function $v\in C^{2}(\overline{\Omega })$
and a positive number $\tau $ such that
$(u_{n}, v_{n})$ converges uniformly to
$(\tau /v, v)$ 
by passing to a subsequence if necessary,
and $(v, \tau)$ satisfies
\begin{equation}\label{LNlim1}
\begin{cases}
d_{2}\Delta v+v(a_{2}-c_{2}v)-b_{2}\tau =0
\ \ &\mbox{in}\ \Omega,\\
\partial_{\nu}v=0
\ \ &\mbox{on}\ \partial\Omega,\\
\displaystyle\int_{\Omega}
\dfrac{1}{v}\left(
a_{1}-\dfrac{b_{1}\tau }{v}-c_{1}v\right) =0;
\end{cases}
\end{equation}
\item
there are positive functions $u$,
$w\in C^{2}(\overline{\Omega })$ such that
$(u_{n}, \alpha_{n}v_{n})$ converges uniformly to $(u,w)$
by passing to a subsequence if necessary,
and $(u, w)$ satisfies
\begin{equation}\label{LNlim2}
\begin{cases}
\Delta[\,(d_{1}+w)u\,]+u(a_{1}-b_{1}u)=0
\ \ &\mbox{in}\ \Omega,\\
\Delta [\,(d_{2}+\beta u)w\,] +w(a_{2}-b_{2}u)=0
\ \ &\mbox{in}\ \Omega,\\
\partial_{\nu}u=\partial_{\nu}w=0
\ \ &\mbox{on}\ \partial\Omega.
\end{cases}
\end{equation}
\end{enumerate}
\end{thm}\vspace{5mm}
Thanks to Theorem \ref{LNthm},
one can expect that nonconstant solutions of \eqref{SKT}
can be characterized by those of \eqref{LNlim1} or \eqref{LNlim2}
if $\alpha>0$ is sufficiently large and $\beta\ge 0$ is 
sufficiently small.
Indeed, such perturbations were verified by
\cite{LW1, LN2, LNY2, NWX, WWX, Wu, WX} in various senses.
In the first limiting behavior stated in (i) of Theorem \ref{LNthm},
$u_{n}v_{n}$ approaches a positive constant $\tau$ uniformly in
$\overline{\Omega }$, and thereby, it is natural to expect that
the first limiting system \eqref{LNlim1} can realize the
segregation phenomena of two competing species when 
one of cross-diffusive abilities of two species is very strong.
Once Theorem \ref{LNthm} was revealed by \cite{LN2},
there has been a great progress of study of the first limiting system
\eqref{LNlim1}
(e.g., \cite{KW, LNY, LNY2, MSY, NWX, WWX, Wu, WX, Yo}).
Among other things, Lou, Ni and Yotsutani \cite{LNY}
obtained a global bifurcation structure of 
positive solutions in the one-dimensional case.
In the second limiting behavior stated in (ii) of Theorem \ref{LNthm},
the stationary density $v_{n}$ of the species with small 
cross-diffusive ability shrinks with the order $O(1/\alpha_{n})$
as $\alpha_{n}\to\infty$ since $\alpha_{n}v_{n}$
tends to a positive function $w$.
The author \cite{Ku} obtained a global bifurcation structure of
positive nonconstant solutions of \eqref{LNlim2}
in a special case when $\beta =0$ and 
$\Omega$ is a one-dimensional interval
Furthermore, Li and Wu \cite{LW1} investigated the instability
of positive nonconstant solutions near the bifurcation
point.
By gathering information on solutions of \eqref{LNlim1}
or \eqref{LNlim2} obtained in above mentioned papers,
we have a reasonable conjecture on the bifurcation structure
of \eqref{SKT} with large $\alpha$ and small $\beta\ge 0$
that the set of positive nonconsant solutions form 
bifurcation branches of saddle node type, and moreover,
the upper branches can be approximated by solutions of 
the first limiting system \eqref{LNlim1}, whereas
the lower branches can be characterized by solutions of
the second limiting system \eqref{LNlim2} by regarding
$a_{2}$ as a bifurcation parameter
(see \cite[Figure 1]{Ku}).
In addition, we note that only the second limiting situation as (ii) occurs
under homogeneous Dirichlet boundary conditions
(\cite{KY1, KY2}).

The purpose of this paper is to study the asymptotic
behavior of solutions of \eqref{SKT}
as both $\alpha$ and $\beta$ tend to infinity
with $\alpha/\beta$ approaching a positive number.
Ecologically, we expect that such a study can reveal
the mathematical mechanism of segregation of 
two competing species when the cross-diffusive abilities
of both species are strong.
To this end, we obtain the a priori $L^{\infty}$ estimate
of all solutions of \eqref{SKT}
as follows:
\begin{thm}\label{Linfthm}
For any small $\eta >0$, 
there exists a positive constant $C=C(\eta, d_{i},a_{i}, b_{i}, c_{i})$
such that
if $\alpha>0$ and $\beta >0$
satisfy $\eta\le\alpha/\beta\le 1/\eta$,
then any solution $(u,v)$ of \eqref{SKT}
satisfies 
$$\max_{x\in\overline{\Omega}}u(x)\le C
\qquad\mbox{and}\qquad
\max_{x\in\overline{\Omega}}v(x)\le C.$$
\end{thm}
Our approach of the proof is based on the maximum principle.
In view of some papers studying \eqref{SKT},
it can be said that
a usual method in considering the a priori $L^{\infty}$ 
estimate is to employ
the following change of variables 
\begin{equation}\label{change}
\phi (x)=\biggl(1+\dfrac{\alpha}{d_{1}}v\biggr)u,
\qquad
\psi (x)=\biggl(1+\dfrac{\beta}{d_{2}}u\biggr)v,
\end{equation}
which reduces
the quasilinear system \eqref{SKT} to
the semilinear one as follows
\begin{equation}\label{semi}
\begin{cases}
d_{1}\Delta \phi+
f(u,v)=0\qquad&\mbox{in}\ \Omega,\\
d_{2}\Delta \psi+
g(u,v)=0\qquad&\mbox{in}\ \Omega,\\
\phi\ge 0,\quad\psi\ge 0\qquad&\mbox{in}\ \Omega,\\
\partial_{\nu}\phi=\partial_{\nu}\psi=0
\qquad&\mbox{on}\ \partial\Omega,
\end{cases}
\end{equation}
where $(u,v)$ in reaction terms is regarded
as a function of $(\phi,\psi)$ determined by 
\eqref{change}.
A typical application of the maximum principle to 
the first equation of \eqref{semi}
enables us to know the nonnegativity of 
$f(u,v)$ 
at the maximum point of $\phi$.
However, obviously this maximum point is different 
from a maximum point of $u$,
and then, such a difference
often makes our construction of 
an $L^{\infty}$ bound of solutions be difficult.
In \cite[Theorem 2.3]{LN2}, an exquisite combination of the above
maximum principle approach
and the Harnack inequality established the uniform $L^{\infty}$
estimate of any solution $(u,v)$ in the case when $N\le 3$
and
$\alpha>0$ is arbitrary but $\beta\ge 0$ is restricted to be small.
The restriction $N\le 3$ comes from
the Sobolev embedding theorem for the use of the Harnack inequality.

In this paper,
in order to get the uniform $L^{\infty}$ estimate
of any solution in a case when $\alpha> 0$ and
$\beta> 0$ are arbitrary as long as $\eta\le\alpha/\beta\le 1/\eta $,
we employ a different approach 
(without the change of variables \eqref{change})
to reduce \eqref{SKT} to the following form:
\begin{equation}\label{new}
\begin{cases}
(d_{1}d_{2}+d_{1}\beta u+d_{2}\alpha v)\Delta u
+2d_{2}\alpha\nabla u\!\cdot\!\nabla v+uF(u,v; \alpha,\beta )=0
\qquad&\mbox{in}\ \Omega,\\
(d_{1}d_{2}+d_{1}\beta u+d_{2}\alpha v)\Delta v
+2d_{1}\beta\nabla u\!\cdot\!\nabla v+vG(u,v; \alpha,\beta )=0
\qquad&\mbox{in}\ \Omega,\\
u\ge 0,\quad v\ge 0\qquad&\mbox{in}\ \Omega,\\
\partial_{\nu}u=\partial_{\nu}v=0
\qquad&\mbox{on}\ \partial\Omega,
\end{cases}
\end{equation}
where
\begin{equation}\label{FGdef}
\begin{cases}
F(u,v; \alpha,\beta ):=
(d_{2}+\beta u)(a_{1}-b_{1}u-c_{1}v)-\alpha v(a_{2}-b_{2}u-c_{2}v),\\
G(u,v; \alpha,\beta ):=
-\beta u(a_{1}-b_{1}u-c_{1}v)+(d_{1}+\alpha v)(a_{2}-b_{2}u-c_{2}v).
\end{cases}
\end{equation}
For \eqref{new}, as explained in the next section,
the maximum principle leads to the nonnegativity 
of $F$ (resp. $G$)
at the maximum point of $u$ (resp. $v$).
In the proof, we make use of a fact that
if $(u,v)\in\mathbb{R}^{2}_{+}$ satisfies $F(u,v,\alpha,\beta )\ge 0$ 
and
$(d_{2}b_{1}+d_{1}b_{2})u+(d_{2}c_{1}+d_{1}c_{2})v>d_{2}a_{1}+d_{1}a_{2}$,
then $G(u,v,\alpha,\beta )<0$.
By the combination of this fact and
a levelset analysis for 
$F$ and $G$,
the proof of Theorem \ref{Linfthm}
will be carried out.
Since our proof does not use the Harnack
inequality as well as the Sobolev embedding theorem,
then Theorem \ref{Linfthm} does not require
any restriction on the dimension number $N$. 

Thanks to Theorem \ref{Linfthm},
we can treat the asymptotic analysis for solutions of \eqref{SKT}
as $\alpha\to\infty$ and
$\beta\to\infty$ with
$\alpha /\beta \to \gamma $
for some $\gamma >0$.
We obtain the following limiting systems in such a
{\it full cross-diffusion limit}.
\begin{thm}\label{limthm}
Suppose that
$a_{1}/a_{2}\neq b_{1}/b_{2}$
and
$a_{1}/a_{2}\neq c_{1}/c_{2}$.
Let $\{(u_{n},v_{n})\}$ be any sequence of positive 
nonconstant solutions of \eqref{SKT} with
$\alpha=\alpha_{n}\to\infty$,
$\beta=\beta_{n}\to\infty$
and
$\gamma_{n}:=\alpha_{n}/\beta_{n}\to\gamma >0$
as $n\to\infty$.
Then either of the following two situations occurs,
passing to a subsequece if necessary;
\begin{enumerate}[(i)]
\item
there exist a positive function $u\in C^{2}(\overline{\Omega })$ and
a positive number $\tau $ 
such that
$$\lim_{n\to\infty}(u_{n},v_{n})=\biggl(u,\dfrac{\tau}{u}\biggr)
\qquad\mbox{in}\ C^{1}(\overline{\Omega})\times C^{1}(\overline{\Omega }),$$
and
$w_{n}(x):=d_{1}u_{n}(x)-\gamma_{n}d_{2}v_{n}(x)$ satisfies
$$
\lim_{n\to\infty}w_{n}=w
\qquad\mbox{in}\ C^{1}(\overline{\Omega })$$
with some function
$w\in C^{2}(\overline{\Omega })$ 
satisfying
\begin{equation}\label{IS}
\begin{cases}
\Delta w+f\biggl(\dfrac{\sqrt{w^{2}+4\gamma d_{1}d_{2}\tau}+w}{2d_{1}},
\dfrac{\sqrt{w^{2}+4\gamma d_{1}d_{2}\tau}-w}{2\gamma d_{2}}\biggr)
\vspace{1mm}\\
-\gamma g\biggl(\dfrac{\sqrt{w^{2}+4\gamma d_{1}d_{2}\tau}+w}{2d_{1}},
\dfrac{\sqrt{w^{2}+4\gamma d_{1}d_{2}\tau}-w}{2\gamma d_{2}}\biggr)=0
\qquad&\mbox{in}\ \Omega,\vspace{1mm} \\
\partial_{\nu}w=0\qquad&\mbox{on}\ \partial\Omega,\vspace{1mm}\\
\displaystyle\int_{\Omega }f\biggl(\dfrac{\sqrt{w^{2}+4\gamma d_{1}d_{2}\tau}+w}{2d_{1}},
\dfrac{\sqrt{w^{2}+4\gamma d_{1}d_{2}\tau}-w}{2\gamma d_{2}}\biggr)=0
\end{cases}
\end{equation}
and
$$
\biggl(u,\dfrac{\tau}{u}\biggr)=\biggl(\dfrac{\sqrt{w^{2}+4\gamma d_{1}d_{2}\tau}+w}{2d_{1}},
\dfrac{\sqrt{w^{2}+4\gamma d_{1}d_{2}\tau}-w}{2\gamma d_{2}}\biggr);
$$
\item
there exist nonnegative functions
$u$,
$v\in C(\overline{\Omega})$ such that
$uv=0$
in $\Omega$,
$$
\lim_{n\to\infty}(u_{n},v_{n})=(u,v)
\qquad\mbox{uniformly in}\ \overline{\Omega}$$
and $w_{n}(x):=d_{1}u_{n}(x)-\gamma_{n}d_{2}v_{n}(x)$ satisfies
$$
\lim_{n\to\infty}w_{n}=w
\qquad\mbox{in}\ C^{1}(\overline{\Omega })$$
with some sign-changing function $w$
satisfying
\begin{equation}\label{CS}
\begin{cases}
\Delta w+f\biggl(\dfrac{w_{+}}{d_{1}},\dfrac{w_{-}}{\gamma d_{2}}\biggr)
-\gamma g\biggl(\dfrac{w_{+}}{d_{1}},\dfrac{w_{-}}{\gamma d_{2}}\biggr)=0
\qquad&\mbox{in}\ \Omega,\\
\partial_{\nu}w=0
\qquad&\mbox{on}\ \partial\Omega,
\end{cases}
\end{equation}
and
$$
(u,v)=\biggl(\dfrac{w_{+}}{d_{1}},
\dfrac{w_{-}}{\gamma d_{2}}\biggr),
$$
where
$w_{+}:=\max\{w,0\}$ and
$w_{-}:=-\min\{w,0\}\ge 0$.
\end{enumerate}
\end{thm}

It should be noted that 
a formal observation by Kan-on \cite{Ka3} implied the existence of 
a nonstationary version of the limiting system \eqref{IS}.
Thus it can be said that Theorem \ref{limthm} supports his observation
by a rigorous proof as far as the stationary problem.
In both situations (i) and (ii) of Theorem \ref{limthm},
$u_{n}v_{n}$ approaches some constant $\tau$ as
$\alpha_{n}$,
$\beta_{n}\to\infty$
and
$\alpha_{n}/\beta_{n}\to\gamma$.
Ecologically, this fact enables us to expect the segregation
of competing species occurs when cross-diffusive abilities
of both species are strong to the same degree.
In the limiting case (i) of Theorem \ref{limthm}, since $\tau>0$,
a typical expected ecological situation is so that
the high (resp. low) density area of $u$ is 
the low (resp. high) density area of $v$
(the {\it incomplete segregation}).
In the other limiting case (ii) of Theorem \ref{limthm},
since $\tau=0$,
living areas of two competing species completely segregate
each other (the {\it complete segregation}).
It is known that \eqref{CS} appears also in the fast reaction limit
of the Lotka-Volterra competition model
(namely,
in the limiting case as $c_1$, $b_2\to\infty$ and $c_1/b_2$ tends to some positive number
in \eqref{SKT} with $\alpha=\beta=0$), and then,
there are several papers discussing 
the existence of nonconstant solutions of \eqref{CS} and related issues
(e.g., \cite{DD, DHMP, DZ, HY, Ka2}).

The contents of this paper is as follows:
In Section 2, we derive a uniform $L^{\infty}$ estimate
of all solutions of \eqref{SKT} to prove Theorem \ref{Linfthm}.
In Section 3, we study the asymptotic behavior of 
solutions of \eqref{SKT} as $\alpha_{n}$, $\beta_{n}\to\infty$
and $\alpha_{n}/\beta_{n}\to\gamma$ to prove Theorem \ref{limthm}.
In Section 4, we discuss the existence of nonconstant solutions
of the limiting system \eqref{IS} from the bifurcation viewpoint.

Throughout this paper, 
the usual norms of the spaces $L^{p}(\Omega )$
for $p\in [\,1,\infty )$ and $L^{\infty }(\Omega )$
are denoted by
$$
\|u\|_{p}:=
\left(\displaystyle\int_{\Omega }|u(x)|^{p}dx\right)^{1/p},
\ \ \ \ 
\|u \|_{\infty }:=\mbox{ess.}\sup_{x\in\overline{\Omega }}|u(x)|.
$$
Hence $\|u\|_{\infty}=\max_{x\in\overline{\Omega}}$
in a case when $u\in C(\overline{\Omega })$.

\section{Uniform boundedness of steady states}
This section is devoted to the proof of 
Theorem \ref{Linfthm}.
Our strategy of the proof is to employ a maximum principle
approach for a reduction form \eqref{new}.
We begin with the reduction.
\begin{lem}
If $(u,v)$ is a solution of \eqref{SKT},
then $(u,v)$ is a solution of \eqref{new}.
\end{lem}

\begin{proof}
Let $(u,v)$ be any solution of \eqref{SKT}.
By expanding the cross-diffusion terms,
one can see that the elliptic equations of \eqref{SKT}
are expressed as
$$
\begin{cases}
(d_{1}+\alpha v)\Delta u+2\alpha\nabla u\!\cdot\!\nabla v
+\alpha u\Delta v+f(u,v)=0\qquad&\mbox{in}\ \Omega,\\
(d_{2}+\beta u)\Delta v+2\beta\nabla u\!\cdot\!\nabla v
+\beta v\Delta u+g(u,v)=0\qquad&\mbox{in}\ \Omega.
\end{cases}
$$
Plugging the expression of $\Delta v$ from
the second equation into the first equation,
we obtain the first equation of \eqref{new}.
A similar procedure also gives the second equation of 
\eqref{new}.
\end{proof}
Applications of the following maximum principle
to \eqref{new}
will be useful in the proof of Theorem \ref{Linfthm}.
\begin{lem}[e.g., \cite{LN2}]\label{MPlem}
Suppose that $h\in C(\overline{\Omega}\times\mathbb{R})$
and $\boldsymbol{B}\in C(\overline{\Omega};\mathbb{R}^{N})$.
Then the followings (i) and (ii) hold true:
\begin{enumerate}[(i)]
\item
If $\underline{u}\in C^{2}(\Omega )
\cap C^{1}(\overline{\Omega })$ satisfies
$$
\Delta \underline{u}+
\boldsymbol{B}(x)\!\cdot\!\nabla\underline{u}+h(x,\underline{u})\ge 0
\quad\mbox{in}\ \Omega,\qquad
\partial_{\nu}\underline{u}\le 0
\quad\mbox{on}\ \partial\Omega,
$$
and $\underline{u}(x_{0})=\|\underline{u}\|_{\infty}$,
then $h(x_{0},\underline{u}(x_{0}))\ge 0$. \vspace{1mm}
\item
If $\overline{u}\in C^{2}(\Omega )
\cap C^{1}(\overline{\Omega })$ satisfies
$$
\Delta \overline{u}+
\boldsymbol{B}(x)\!\cdot\!\nabla\overline{u}+h(x,\overline{u})\le 0
\quad\mbox{in}\ \Omega,\qquad
\partial_{\nu}\overline{u}\ge 0
\quad\mbox{on}\ \partial\Omega,
$$
and $\overline{u}(x_{0})=\min_{x\in\overline{\Omega}}\overline{u}(x)$,
then $h(x_{0},\overline{u}(x_{0}))\le 0$.
\end{enumerate}
\end{lem}

For the application of Lemma \ref{MPlem} to \eqref{new},
we need to know the profile of $F(u,v;\alpha,\beta)$
defined by \eqref{FGdef}.
The next lemma yields information on the zero levelset of 
$F(u,v; \alpha, \beta )$.
\begin{lem}\label{Fprolem}
Suppose that $\alpha >0$ and $\beta >0$.
Then the followings (i) and (ii) hold true.
\begin{enumerate}[(i)]
\item
If $u>a_{1}/b_{1}$, then there exists a positive number
$V(u;\alpha, \beta)$ such that
\begin{equation}\label{tatecut}
F(u,v;\alpha,\beta )\begin{cases}
<0\quad\mbox{for}\ 0<v<V(u;\alpha,\beta),\\
=0\quad\mbox{for}\ v=V(u;\alpha,\beta),\\
>0\quad\mbox{for}\ v>V(u;\alpha,\beta).
\end{cases}
\end{equation}
\item
Define
\begin{equation}\label{v0def}
\widetilde{v}_{0}(\alpha )
:=\begin{cases}
0\quad &\mbox{if}\ c_{1}/c_{2}<a_{1}/a_{2}
\ \mbox{and}\ \underline{\alpha}<\alpha<\overline{\alpha},\\
\dfrac{\alpha a_{2}+d_{2}c_{1}+
\sqrt{(\alpha a_{2}+d_{2}c_{1})^{2}-4\alpha d_{2}a_{1}c_{2}}}
{2\alpha c_{2}}
\quad &\mbox{otherwise},
\end{cases}
\end{equation}
where
$$
\begin{cases}
\underline{\alpha}:=
\dfrac{d_{2}(2a_{1}{c_{2}}-a_{2}c_{1})-
2d_{2}\sqrt{a_{1}c_{2}(a_{1}c_{2}-a_{2}c_{1})}}
{a_{2}^{\,2}},\\
\overline{\alpha}:=
\dfrac{d_{2}(2a_{1}{c_{2}}-a_{2}c_{1})+
2d_{2}\sqrt{a_{1}c_{2}(a_{1}c_{2}-a_{2}c_{1})}}
{a_{2}^{\,2}}.
\end{cases}
$$
If $v>\widetilde{v}_{0}(\alpha )$,
then there exists a positive number
$U(v;\alpha, \beta)$ such that
\begin{equation}\label{yokocut}
F(u,v;\alpha,\beta )\begin{cases}
>0\quad\mbox{for}\ 0<u<U(v;\alpha,\beta),\\
=0\quad\mbox{for}\ u=U(v;\alpha,\beta).\\
<0\quad\mbox{for}\ u>U(v;\alpha,\beta).
\end{cases}
\end{equation}
\end{enumerate}
\end{lem}

\begin{proof}
(i)\ 
We first observe the sign of $F$ on the half line
$\{(u,0)\,:\,u>0\}$ on $u$ axis as follows:
$$
F(u,0; \alpha, \beta )
=(d_{2}+\beta u)(a_{1}-b_{1}u)
\begin{cases}
>0\quad \mbox{for}\ 0<u<a_{1}/b_{1},\\
<0\quad\mbox{for}\ u>a_{1}/b_{1}.
\end{cases}
$$
Next, for each fixed $u>a_{1}/b_{1}$,
we investigate the profile of the function
$v\mapsto F(u,v;\alpha,\beta)$
(regarded as a function with respect to $v>0$).
By the form of the quadratic function
$$
F(u,v;\alpha,\beta )=
\alpha c_{2}v^{2}
-\{c_{1}(d_{2}+\beta u)+\alpha (a_{2}-b_{2}u)\}v+
(d_{2}+\beta u)(a_{1}-b_{1}u)
$$
and the fact that
$F(u,0; \alpha, \beta )<0$ for any fixed $u>a_{1}/b_{1}$,
we obtain \eqref{tatecut} with
\begin{equation}
\begin{split}
&V(u;\alpha, \beta)\\
&=\dfrac{c_{1}(d_{2}+\beta u)+\alpha (a_{2}-b_{2}u)+
\sqrt{\{c_{1}(d_{2}+\beta u)+\alpha (a_{2}-b_{2}u)\}^{2}
-4\alpha c_{2}(d_{2}+\beta u)(a_{1}-b_{1}u)}}
{2\alpha c_{2}}.
\end{split}
\nonumber
\end{equation}

(ii)\
Following a similar argument, we first check the sign
of $F$ on the half line
$\{(0,v)\,:\,v>0\}$ on $v$ axis.
By virtue of
$$
F(0,v;\alpha, \beta )=\alpha c_{2}v^{2}-(\alpha a_{2}+d_{2}c_{1})v
+d_{2}a_{1},$$
a straightforward computation enables us to check that if $c_{1}/c_{2}< a_{1}/a_{2}$ and 
$\underline{\alpha}<\alpha<\overline{\alpha}$, then
$F(0,v;\alpha,\beta )> 0$ for any $v>0$;
otherwise,
$$
F(0,v;\alpha,\beta)\begin{cases}
>0\quad &\mbox{for}\ v\in (0,\underline{v}_{0}(\alpha))\cup
(\overline{v}_{0}(\alpha ),\infty),\\
<0\quad &\mbox{for}\ v\in
(\underline{v}_{0}(\alpha ), \overline{v}_{0}(\alpha )),
\end{cases}
$$
where
$$
\begin{cases}
\underline{v}_{0}(\alpha )=
\dfrac{\alpha a_{2}+d_{2}c_{1}-\sqrt{
(\alpha a_{2}+d_{2}c_{1})^{2}-4\alpha d_{2}a_{1}c_{2}}}
{2\alpha c_{2}},\vspace{1mm} \\
\overline{v}_{0}(\alpha )=
\dfrac{\alpha a_{2}+d_{2}c_{1}+\sqrt{
(\alpha a_{2}+d_{2}c_{1})^{2}-4\alpha d_{2}a_{1}c_{2}}}
{2\alpha c_{2}}.
\end{cases}
$$
Hence it follows that $F(0,v;\alpha, \beta)>0$ for
$v>\widetilde{v}_{0}(\alpha )$, where
$$
\widetilde{v}_{0}(\alpha ):=\begin{cases}
0\quad &\mbox{if}\ c_{1}/c_{2}<a_{1}/a_{2}
\ \mbox{and}\ \underline{\alpha}<\alpha<\overline{\alpha},\\
\overline{v}_{0}(\alpha )
\quad &\mbox{otherwise}.
\end{cases}
$$
Next, for any fixed $v>\widetilde{v}_{0}(\alpha)$,
we check the profile of $u\mapsto F(u,v;\alpha,\beta)$.
Since $F(0,v;\alpha,\beta )>0$ for 
$v>\widetilde{v}_{0}(\alpha )$, then we fix such $v$ arbitrarily and
regard
$$
u\mapsto F(u,v;\alpha,\beta )=
-\beta b_{1}u^{2}+\{\beta(a_{1}-c_{1}v)+\alpha b_{2}v-d_{2}b_{1}\}u
+ac_{2}v^{2}-(\alpha a_{2}+d_{2}c_{1})v+d_{2}a_{1}$$
as a 
quadratic function with respect to $u>0$
to obtain \eqref{yokocut} with
\begin{equation}\label{Udef}
\begin{split}
U(v;\alpha,\beta)
=&\dfrac{1}{2\beta b_{1}}
\biggl(
(\alpha b_{2}-\beta c_{1})v+\beta a_{1}-d_{2}b_{1}\\
&+
\sqrt{\{(\alpha b_{2}-\beta c_{1})v+\beta a_{1}-d_{2}b_{1}\}^{2}
+4\beta b_{1}\{
\alpha c_{2}v^{2}-(\alpha a_{2}+d_{2}c_{1})v+d_{2}a_{1}\}
}\biggr).
\end{split}
\end{equation}
Then we complete the proof of Lemma \ref{Fprolem}.
\end{proof}
The next lemma is an elementary but a key property
for the proof of Theorem \ref{Linfthm}.
To state the property, we define an unbounded region
$\varSigma$ by
$$\varSigma:=\{
(u,v)\in\mathbb{R}^{2}_{+}\,:\,
d_{2}a_{1}+d_{1}a_{2}-(d_{2}b_{1}+d_{1}b_{2})u-(d_{2}c_{1}+d_{1}c_{2})v<0\}.
$$
It should be noted that
$\varSigma$ is independent of $\alpha$ and $\beta $.
\begin{lem}\label{keylem}
If $(u,v)\in\varSigma$ satisfies
$F(u,v;\alpha, \beta )\ge 0$, then $G(u,v,\alpha,\beta )< 0$.
\end{lem}

\begin{proof}
Since \eqref{FGdef} yields
$$
F(u,v;\alpha, \beta )+G(u,v;\alpha,\beta )
=d_{2}a_{1}+d_{1}a_{2}-(d_{2}b_{1}+d_{1}b_{2})u-(d_{2}c_{1}+d_{1}c_{2})v,$$
the desired property follows.
\end{proof}
By Lemmas \ref{MPlem}-\ref{keylem}, 
we shall accomplish the proof of Theorem \ref{Linfthm}.
\begin{proof}[Proof of Theorem \ref{Linfthm}]
For any small $\eta >0$,
let $\alpha>0$ and $\beta >0$ satisfy
\begin{equation}\label{condi}
\eta\le\dfrac{\alpha}{\beta}\le\dfrac{1}{\eta}.
\end{equation}
We first discuss the case when $0<\alpha\le\eta$.
In this case, \eqref{condi} implies $\beta \le 1$.
Then we can use an estimate obtained by Lou and Ni
\cite[Lemma 2.3]{LN1} to know 
$$
\|u\|_{\infty}\le
C_{1}\biggl(1+\dfrac{\alpha}{d_{1}}\biggr)\le
C_{1}\biggl(1+\dfrac{\eta}{d_{1}}\biggr)
\quad\mbox{and}\quad
\|v\|_{\infty}\le
C_{1}\biggl(1+\dfrac{\beta}{d_{2}}\biggr)\le
C_{1}\biggl(1+\dfrac{1}{d_{2}}\biggr)
$$
with some positive constant
$C_{1}=C_{1}(a_{i},b_{i},c_{i})$.
Similarly, also in the case $0<\beta\le\eta$,
their result leads to
$$
\|u\|_{\infty}\le
C_{1}\biggl(1+\dfrac{\alpha}{d_{1}}\biggr)\le
C_{1}\biggl(1+\dfrac{1}{d_{1}}\biggr)
\quad\mbox{and}\quad 
\|v\|_{\infty}\le
C_{1}\biggl(1+\dfrac{\beta}{d_{2}}\biggr)\le
C_{1}\biggl(1+\dfrac{\eta}{d_{2}}\biggr).
$$
Then, for the sake of the proof of Theorem \ref{Linfthm},
we may assume
\begin{equation}\label{condi2}
\alpha>\eta\quad\mbox{and}\quad\beta>\eta,
\end{equation}
in addition to \eqref{condi}.
Our first aim is to prove that 
any solution $(u,v)$ of \eqref{SKT} satisfies
\begin{equation}\label{aim}
\|u\|_{\infty}\le
\max\biggl\{
\dfrac{a_{1}}{b_{1}},
U\biggl(\max\biggl\{\dfrac{d_{2}a_{1}+d_{1}a_{2}}{d_{2}b_{1}+d_{1}b_{2}},
\widetilde{v}_{0}(\alpha )\biggr\}; \alpha, \beta \biggr)\biggr\},
\end{equation}
where $U>0$ and $\widetilde{v}_{0}(\alpha )\ge 0$ 
are numbers represented as \eqref{Udef} and \eqref{v0def}, respectively.
Here we recall Lemma \ref{Fprolem} to note that
if $u>\max\{a_{1}/b_{1}, U(\widetilde{v}_{0}(\alpha );
\alpha, \beta )\}$, then 
$u=U(v;\alpha,\beta )$ and
$v=V(u;\alpha,\beta)$ are inverses of each other,
and these functions are monotone increasing with
$$
\lim_{u\to\infty}V(u;\alpha, \beta )=\infty
\quad\mbox{and}\quad
\lim_{v\to\infty}U(v;\alpha, \beta )=\infty,
$$
where $U(\widetilde{v}_{0}(\alpha );
\alpha, \beta ):=\lim_{v\downarrow\widetilde{v}_{0}(\alpha )}
U(v;\alpha,\beta )$.
In order to show \eqref{aim} for any solution $(u,v)$ of \eqref{SKT},
we employ a proof by contradiction.
Suppose for contradiction that \eqref{SKT}
admits a solution $(u,v)$ of \eqref{SKT} satisfying
\begin{equation}\label{contra}
\|u\|_{\infty}>
\max\biggl\{
\dfrac{a_{1}}{b_{1}},
U\biggl(\max\biggl\{\dfrac{d_{2}a_{1}+d_{1}a_{2}}{d_{2}b_{1}+d_{1}b_{2}},
\widetilde{v}_{0}(\alpha )\biggr\}; \alpha, \beta \biggr)\biggr\}.
\end{equation}
Let $x^{*}\in\overline{\Omega}$ be a maximum point of $u$,
that is,
$\|u\|_{\infty}=u(x^{*})$.
Since $(u,v)$ satisfies \eqref{new},
then we know
\begin{equation}
F(u(x^{*}), v(x^{*});\alpha, \beta )\ge 0
\nonumber
\end{equation}
by applying (i) of Lemma \ref{MPlem} to
the first equation of \eqref{new}.
Together with $u(x^{*})>a_{1}/b_{1}$ from \eqref{contra},
we know from
(i) of Lemma \ref{Fprolem} that
\begin{equation}\label{vVstar}
v(x^{*})\ge V(u(x^{*});\alpha,\beta ).
\end{equation}
If $\widetilde{v}_{0}(\alpha )\ge (d_{2}a_{1}+d_{1}a_{2})/
(d_{2}b_{1}+d_{1}b_{2})$,
then it follows from \eqref{contra} that
\begin{equation}\label{udomain}
    u(x^{*})>\dfrac{a_{1}}{b_{1}}\quad\mbox{and}\quad u(x^{*})>U(\widetilde{v}_{0}(\alpha ); \alpha, \beta ).
\end{equation}
Here we recall that the function $v\mapsto V(v;\alpha, \beta )$
is monotone increasing for $v>\widetilde{v}_{0}(\alpha )$ as well as
the function
$u\mapsto U(u;\alpha, \beta )$ is monotone increasing
for $u>\max\{a_{1}/b_{1}, U(\widetilde{v}_{0}(\alpha ); \alpha, \beta )\}$
since $V(u.\alpha, \beta)$ is an inverse function of 
$U(v;\alpha, \beta )$.
Then \eqref{udomain} leads to
$V(u(x^{*});\alpha,\beta )>\widetilde{v}_{0}(\alpha )$, and thereby,
\begin{equation}\label{vstar}
v(x^{*})>\widetilde{v}_{0}(\alpha )\ge
\dfrac{d_{2}a_{1}+d_{1}a_{2}}{d_{2}b_{1}+d_{1}b_{2}}.
\end{equation}
Then, in the case when
$\widetilde{v}_{0}(\alpha )\ge (d_{2}a_{1}+d_{1}a_{2})/(d_{2}b_{1}+d_{1}b_{2})$,
by virtue of  \eqref{udomain}, \eqref{vstar} and
the monotone increasing property of $U(v;\alpha,\beta)$
as well as $V(u;\alpha,\beta)$,
we know from Lemma \ref{Fprolem} that
the strip region
$$\mathcal{R}:=\{(u,v)\in\mathbb{R}^{2}_{+}\,:\,
0<u\le u(x^{*}),\ v\ge v(x^{*})\}$$
is contained in the positive region of $F$, that is,
$$
\mathcal{R}\subset\{F>0\},
$$
where
$\{F>0\}:=\{(u,v)\in\mathbb{R}^{2}_{+}\,:\,F(u,v;\alpha,\beta)>0\}$.
On the other hand,
if $\widetilde{v}_{0}(\alpha )< (d_{2}a_{1}+d_{1}a_{2})/
(d_{2}b_{1}+d_{1}b_{2})$,
then \eqref{contra} implies
$$
u(x^{*})>a_{1}/b_{1}\quad\mbox{and}\quad
u(x^{*})>U\biggl(\dfrac{d_{2}a_{1}+d_{1}a_{2}}
{d_{2}b_{1}+d_{1}b_{2}};\alpha, \beta\biggr).$$
Since $u\mapsto V(u;\alpha, \beta )$ is monotone increasing
for $V>\widetilde{v}_{0}(\alpha )$ and it is an inverse function of 
$U(v;\alpha, \beta )$, then
$V(u(x^{*});\alpha,\beta )>(d_{2}a_{1}+d_{1}a_{2})/
(d_{2}b_{1}+d_{1}b_{2})$.
With \eqref{vVstar}, one can see 
\begin{equation}\label{vstar2}
v(x^{*})>
\dfrac{d_{2}a_{1}+d_{1}a_{2}}{d_{2}b_{1}+d_{1}b_{2}}>
\widetilde{v}_{0}(\alpha ).
\end{equation}
Similarly, the monotone increasing property of $U(v;\alpha,\beta)$
or $V(u;\alpha,\beta)$ implies  
$\mathcal{R}\subset \{F>0\}$.
Together with \eqref{vstar} and \eqref{vstar2},
we can see that our assumption \eqref{contra} leads to
\begin{equation}\label{RinG}
\mathcal{R}\subset\{F>0\}\cap\varSigma.
\end{equation}
Next, let $y^{*}\in\overline{\Omega }$ be a maximum point of 
$v$, namely,
$\|v\|_{\infty}=v(y^{*})$.
Then applying (i) of Lemma \ref{MPlem} to the second equation
of \eqref{new}, one can see that
\begin{equation}\label{mp}
G(u(y^{*}), v(y^{*});\alpha, \beta )\ge 0.
\end{equation} 
Since $0\le u(y^{*})\le u(x^{*})=\|u\|_{\infty}$ and 
$v(x^{*})\le v(y^{*})=\|v\|_{\infty}$,
then $(u(y^{*}),v(y^{*}))\in\mathcal{R}$.
Then \eqref{RinG} implies that
$$
F(u(y^{*}),v(y^{*});\alpha, \beta)>0
\quad\mbox{and}\quad
(u(y^{*}), v(y^{*}))\in\varSigma.
$$
Therefore, Lemma \ref{keylem} leads to
$G(u(y^{*}), v(y^{*});\alpha, \beta )< 0$.
Hence this contradicts \eqref{mp}.
Consequently, the above proof by contradiction enables us
to conclude that all solutions of \eqref{SKT} satisfy \eqref{aim}.

Next we shall find a positive constant
$C_{2}=C_{2}(\eta, d_{i}, a_{i}, b_{i}, c_{i})$ such that
\begin{equation}\label{C2}
U\biggl(\max\biggl\{\dfrac{d_{2}a_{1}+d_{1}a_{2}}{d_{2}b_{1}+d_{1}b_{2}},
\widetilde{v}_{0}(\alpha )\biggr\}; \alpha, \beta \biggr)\le C_{2}
\end{equation}
for any $\alpha$,
$\beta>\eta$ with
$\eta\le\alpha/\beta\le 1/\eta$.
To this end,
we recall $U$ and $\widetilde{v}_{0}(\alpha )$ defined by
\eqref{Udef} and \eqref{v0def} are expressed as
\begin{equation}
\begin{split}
U(v; \alpha, \beta )=
&\dfrac{1}{2b_{1}}\Biggl(
(rb_{2}-c_{1})v+a_{1}-\dfrac{d_{2}b_{1}}{\beta }\\
&
\sqrt{
\biggl\{(rb_{2}-c_{1})v+a_{1}
-\dfrac{d_{2}b_{1}}{\beta} 
\biggr\}^{2}+4b_{1}\biggl\{
rc_{2}v^{2}-\biggl(ra_{2}+\dfrac{d_{2}c_{1}}{\beta }\biggr)v
+\dfrac{d_{2}a_{1}}{\beta }\biggr\}}\Biggr),
\end{split}
\nonumber
\end{equation}
where $r:=\alpha/\beta$ and
$\widetilde{v}_{0}(\alpha )=0$ or
$$
\widetilde{v}_{0}(\alpha )=
\dfrac{1}{2c_{2}}\Biggl\{
a_{2}+\dfrac{d_{2}c_{1}}{\alpha}+
\sqrt{\biggl(a_{2}+\dfrac{d_{2}c_{1}}{\alpha}\biggr)^{2}-
\dfrac{4d_{2}a_{1}c_{2}}{\alpha }}\Biggr\}.
$$
Hence these expressions ensure a desired positive constant
$C_{2}=C_{2}(\eta, d_{i}, a_{i}, b_{i}, c_{i})$ satisfying
\eqref{C2}
if $\alpha$,
$\beta>\eta$ and
$\eta\le r\le1/\eta$
(recall \eqref{condi2}).
Consequently, we obtain a positive constant
$C=C(\eta,d_{1},a_{i},b_{i},c_{i})$ such that
$\|u\|_{\infty}\le C$ for any solution $(u,v)$ of \eqref{SKT}.
Obviously, by the same argument replacing
the first equation by the second one in \eqref{SKT},
we can also obtain the desired estimate
of $\|v\|_{\infty}$ for any solution $(u,v)$ of \eqref{SKT}.
Then we complete the proof of Theorem \ref{Linfthm}.
\end{proof}

\section{Full cross-diffusion limit}
In this section, we study the asymptotic behavior of 
nonconstant solutions of \eqref{SKT} as 
$\alpha\to\infty$ and
$\beta\to\infty$ with 
$\alpha/\beta\to \gamma>0$
to prove Theorem \ref{limthm}.
In the proof, the following lemma by Lou and Ni
\cite[Lemma 2.4]{LN2} will be used.
\begin{lem}[\cite{LN2}]\label{avoidlem}
Suppose that $a_{1}/a_{2}\neq b_{1}/b_{2}$,
$a_{1}/a_{2}\neq c_{1}/c_{2}$ and
$\{(u_{n},v_{n})\}$ are positive solutions
of \eqref{SKT} with 
$(d_{1}, d_{2}, \alpha, \beta)=(
d_{1,n}, d_{2,n}, \alpha_{n}, \beta_{n})$.
Assume that $(u_{n},v_{n})\to (u^{*}, v^{*})$
uniformly in $\overline{\Omega}$ as $n\to\infty$
for some nonnegative constants $u^{*}$ and $v^{*}$.
Then, either
$$
\dfrac{b_{2}}{b_{1}}>\dfrac{a_{2}}{a_{1}}>\dfrac{c_{2}}{c_{1}}
\qquad\mbox{or}\qquad
\dfrac{b_{2}}{b_{1}}<\dfrac{a_{2}}{a_{1}}<\dfrac{c_{2}}{c_{1}},
$$
moreover,
$(u^{*},v^{*})$ is the unique root of 
$a_{1}-b_{1}u-c_{1}v=a_{2}-b_{2}u-c_{2}v=0$.
\end{lem}

\begin{proof}[Proof of Theorem \ref{limthm}]
Suppose that $\{(\alpha_{n}, \beta_{n})\}$ is any 
positive sequence satisfying $\alpha_{n}\to\infty$,
$\beta_{n}\to\infty$ and
$\gamma_{n}:=\alpha_{n}/\beta_{n}\to\gamma$
with some positive number $\gamma$.
Let $\{(u_{n}, v_{n})\}$ be positive
solutions of \eqref{SKT} with $(\alpha, \beta )=(\alpha_{n},
\beta_{n})$.
Multiplying the second equation of \eqref{SKT} by $\gamma_{n}$ 
and subtracting the resulting expression from the first equation,
we see that
\begin{equation}\label{wdef}
w_{n}(x):=d_{1}u_{n}(x)-\gamma_{n}d_{2}v_{n}(x)
\end{equation}
satisfies
\begin{equation}\label{weq}
-\Delta w_{n}=f(u_{n}, v_{n})-\gamma_{n}g(u_{n}, v_{n})
\quad\mbox{in}\ \Omega,\qquad
\partial_{\nu}w_{n}=0\quad\mbox{on}\ \partial\Omega.
\end{equation}
In view of the diffusion part of the first equation of \eqref{SKT},
we set
\begin{equation}\label{zdef}
z_{n}(x):=\dfrac{d_{1}}{\alpha_{n}}u_{n}(x)+u_{n}(x)v_{n}(x),
\end{equation}
which satisfies 
\begin{equation}\label{zeq}
-\Delta z_{n}=\dfrac{1}{\alpha_{n}}f(u_{n}, v_{n})
\quad\mbox{in}\ \Omega,\qquad
\partial_{\nu}z_{n}=0
\quad\mbox{on}\ \partial\Omega.
\end{equation}
It is possible to check that
the correspondence of $(u_{n}, v_{n})$ to $(w_{n}, z_{n})$ 
defined by \eqref{wdef} and \eqref{zdef} is one-to-one,
and more precisely,
$(u_{n}, v_{n})$ is expressed as
\begin{equation}\label{uvexp}
\begin{cases}
u_{n}=
\dfrac{1}{2d_{1}}\biggl(
\sqrt{\bigl(w_{n}-\dfrac{d_{1}d_{2}}{\beta_{n}}\bigr)^{2}
+4\gamma_{n}d_{1}d_{2}z_{n}}+w_{n}-\dfrac{d_{1}d_{2}}{\beta_{n}}\biggr),
\vspace{1mm}\\
v_{n}=\dfrac{1}{2\gamma_{n}d_{2}}\biggl(
\sqrt{\bigl(w_{n}+\dfrac{d_{1}d_{2}}{\beta_{n}}\bigr)^{2}
+4\gamma_{n}d_{1}d_{2}z_{n}-\dfrac{4d_{1}d_{2}w_{n}}{\beta_{n}}}-
\bigl(w_{n}+\dfrac{d_{1}d_{2}}{\beta_{n}}\bigr)\biggr).
\end{cases}
\end{equation}
Owing to Theorem \ref{Linfthm}, there exists a positive 
constant $C_{3}=C_{3}(\eta, d_{i}, a_{i}, b_{i}, c_{i})$ such that
$$
\|f(u_{n}, v_{n})\|_{\infty},\ \
\|g(u_{n}, v_{n})\|_{\infty},\ \
\|w_{n}\|_{\infty},\ \
\|z_{n}\|_{\infty}\le C_{3}
$$
for all $n\in\mathbb{N}$.
By applying the elliptic regularity theory (e.g.,\cite{GT})
to \eqref{weq} and \eqref{zeq}, for any $p>1$, we find a a positive 
constant $C_{4}=C_{4}(\eta, d_{i}, a_{i}, b_{i}, c_{i}, p)$ such that
$$
\|w_{n}\|_{W^{2,p}},\ \
\|z_{n}\|_{W^{2,p}}\le C_{4}
$$
for all $n\in\mathbb{N}$.
Therefore, the Sobolev embedding theorem and
the elliptic regularity theory ensure 
$w$,
$z\in W^{2,p}(\Omega )\cap C^{1}(\overline{\Omega })$
for sufficiently large $p>1$
such that
\begin{equation}\label{wzlim}
\lim_{n\to\infty }(w_{n}, z_{n})=
(w,z)\quad\mbox{strongly in}\ 
C^{1}(\overline{\Omega })\times C^{1}(\overline{\Omega })\quad
\mbox{and}\quad
\mbox{weakly in}\ W^{2,p}(\Omega )\times W^{2,p}(\Omega ),
\end{equation}
passing a subsequence if necessary.
Since $\{\|f(u_{n},v_{n})\|_{\infty}\}$ is uniformly bounded 
with respect to $n\in\mathbb{N}$, then setting $n\to\infty$
in \eqref{zeq} implies that
$z(x)$ is a harmonic function in $\Omega$
with $\partial_{\nu }z=0$ on $\partial\Omega$,
and therefore,
$z(x)=\tau $ in $\Omega$ 
with some nonnegative constant $\tau$.
Simultaneously, \eqref{zdef} with $\|u_{n}\|_{\infty}\le C$
leads to
\begin{equation}\label{uvt}
\lim_{n\to\infty}u_{n}v_{n}=\tau\quad
\mbox{uniformly in}\ \overline{\Omega }.
\end{equation}

In a case when $\tau >0$, one can deduce from \eqref{uvexp}
and \eqref{wzlim} that
\begin{equation}\label{unvnlim}
\lim_{n\to\infty}
(u_{n},v_{n})=\biggl(
\dfrac{\sqrt{w^{2}+4\gamma d_{1}d_{2}\tau}+w}{2d_{1}},
\dfrac{\sqrt{w^{2}+4\gamma d_{1}d_{2}\tau}-w}{2\gamma d_{2}}
\biggr)\quad\mbox{in}\ C^{1}(\overline{\Omega})\times
C^{1}(\overline{\Omega }).
\end{equation}
Therefore, together with \eqref{wzlim},
we set $n\to\infty $ in \eqref{weq} to verify that $w\in W^{2,p}(\Omega )$ satisfies
\begin{equation}\label{wlimeq}
\begin{cases}
-\Delta w=
f\biggl(\dfrac{\sqrt{w^{2}+4\gamma d_{1}d_{2}\tau}+w}{2d_{1}},
\dfrac{\sqrt{w^{2}+4\gamma d_{1}d_{2}\tau}-w}{2\gamma d_{2}}\biggr)\\
-\gamma
g\biggl(\dfrac{\sqrt{w^{2}+4\gamma d_{1}d_{2}\tau}+w}{2d_{1}},
\dfrac{\sqrt{w^{2}+4\gamma d_{1}d_{2}\tau}-w}{2\gamma d_{2}}\biggr)
\quad&\mbox{in}\ \Omega,\\
\partial_{\nu}w=0\quad&\mbox{on}\ \partial\Omega.
\end{cases}
\end{equation}
In this case, the Schauder estimate for elliptic equations ensures 
that $w$ is a classical solution of \eqref{wlimeq}.
Since $\tau >0$, then \eqref{uvt} and \eqref{unvnlim} imply
$u(x):=\lim_{n\to\infty}u_{n}(x)>0$
and
$v(x):=\lim_{n\to\infty}v_{n}(x)>0$
for all $x\in\overline{\Omega}$.
Integrating \eqref{zeq} over $\Omega$, we obtain
$\int_{\Omega}f(u_{n},v_{v})=0$
for $n\in\mathbb{N}$.
By \eqref{unvnlim},
the Lebesgue convergence theorem ensures that
$$
\displaystyle\int_{\Omega }f\biggl(\dfrac{\sqrt{w^{2}+4\gamma d_{1}d_{2}\tau}+w}{2d_{1}},
\dfrac{\sqrt{w^{2}+4\gamma d_{1}d_{2}\tau}-w}{2\gamma d_{2}}\biggr)=0.
$$

In the other case when $\tau =0$,
\eqref{uvexp}
and \eqref{wzlim} ensure
\begin{equation}\label{unvnlim2}
\lim_{n\to\infty}
(u_{n},v_{n})=\biggl(
\dfrac{|w|+w}{2d_{1}},
\dfrac{|w|-w}{2\gamma d_{2}}\biggr)
=\biggl(\dfrac{w_{+}}{d_{1}},
\dfrac{w_{-}}{\gamma d_{2}}
\biggr)\quad\mbox{uniformly in}\ 
\overline{\Omega },
\end{equation}
where $w_{+}:=\max\{w,0\}$ and $w_{-}:=-\min\{w,0\}\ge 0$.
Setting $n\to\infty$ in \eqref{weq}, 
we know that $w\in W^{2,p}(\Omega )\cap C^{1}(\overline{\Omega })$
satisfies
\begin{equation}
\begin{cases}
-\Delta w=
f\biggl(\dfrac{w_{+}}{d_{1}},
\dfrac{w_{-}}{\gamma d_{2}}
\biggr)
-\gamma
g\biggl(\dfrac{w_{+}}{d_{1}},
\dfrac{w_{-}}{\gamma d_{2}}
\biggr)
\quad&\mbox{in}\ \Omega,\\
\partial_{\nu}w=0\quad&\mbox{on}\ \partial\Omega.
\end{cases}
\nonumber
\end{equation}
To accomplish the proof of the second
limiting case (ii) stated in Theorem \ref{limthm},
it remains to prove that 
$u(x):=\lim_{n\to\infty}u_{n}(x)$ and $v(x):=\lim_{n\to\infty}v_{n}(x)$
(obtained by \eqref{unvnlim2})
are not constant.
Suppose for contradiction that $u$
(resp. $v$) 
is a positive constant.
Since $uv=0$ in $\Omega$ by \eqref{uvt},
one can see that $v=0$
(resp. $u=0$) 
in $\Omega $.
Namely, \eqref{unvnlim2} implies that
$(u_{n},v_{n})\to (u, 0)$
(resp. $(u_{n},v_{n})\to (0, v)$)
uniformly in $\overline{\Omega}$.
This contradicts Lemma \ref{avoidlem}.
Obviously, $(u,v)=(0,0)$
also contradicts Lemma \ref{avoidlem}.
Consequently, we deduce that $w_{+}$ and $w_{-}$ are not 
identically zero in $\Omega$, in other words,
$w$ is sign-changing.
Then we complete the proof of Theorem \ref{limthm}.
\end{proof}

\section{Existence of nonconstant solutions of limiting systems}
In this section,
as a beginning of study for the limiting system
\eqref{IS} of incomplete segregation,
an existence result of nonconstant solutions will be shown.
In order to state the result,
we note that \eqref{SKT} admits a unique
positive constant solution
\begin{equation}\label{const}
(u^{*},v^{*}):=
\dfrac{1}{b_{2}c_{1}-b_{1}c_{2}}
(a_{2}c_{1}-a_{1}c_{2}, a_{1}b_{2}-a_{2}b_{1})
\end{equation}
in the weak competition case 
$c_{1}/c_{2}<a_{1}/a_{2}<b_{1}/b_{2}$
or the strong competition case 
$b_{1}/b_{2}<a_{1}/a_{2}<c_{1}/c_{2}$,
and therefore,
\eqref{IS}
with
\begin{equation}\label{taus}
\tau=\tau^{*}:=u^{*}v^{*}
\end{equation}
admits a constant solution
$w^{*}:=d_{1}u^{*}-\gamma d_{2}v^{*}$.
In our analysis for \eqref{IS} based on a framework of 
the bifurcation theory,
$w$ and $\tau$ will be regarded as unknowns,
$d_{1}$ will play a role in a bifurcation parameter,
and any other coefficients will be fixed as far as
the weak or the strong competition case.
The next result gives a local curve of 
nonconstant solutions of \eqref{IS},
which bifurcate from $(w^{*}, \tau^{*})$ when the bifurcation parameter $d_{1}$
passes a threshold number.
In what follows, 
all eigenvalues
of $-\Delta$ with homogeneous Neumann boundary condition
on $\partial\Omega$ will be denoted by
$0=\lambda_{0}<\lambda_{1}\le\lambda_{2}\le
\cdots\le\lambda_{j}\le\cdots$
(counting multiplicity).
\begin{thm}\label{bifthm}
Suppose that $c_{1}/c_{2}<a_{1}/a_{2}<b_{1}/b_{2}$ or
$b_{1}/b_{2}<a_{1}/a_{2}<c_{1}/c_{2}$.
Furthermore, assume that
$\lambda_{j}$ is a positive eigenvalue whose eigenspace
is one-dimension.
There exists a small $\eta_{j} >0$ such that
if $0<b_{1},\,c_{2},\,d_{2}<\eta_{j}$,
then there exists $\delta_{j}>0$ such that
nonconstant solutions of \eqref{IS}
bifurcate from the branch of positive constant solutions
$$
\{(d_{1},w, \tau)\,:\,d_{1}>0,\,w=d_{1}u^{*}-\gamma d_{2}v^{*}\,(=:w^{*}(d_{1})),\,\tau=\tau^{*}\}
$$
when $d_{1}$ passes $\delta_{j}$.
More precisely,
in a neighbourhood of
$(d_{1},w,\tau)=(\delta_{j},
w^{*}(\delta_{j}), \tau^{*})\in\mathbb{R}\times
W^{2,p}_{\nu}(\Omega )\times\mathbb{R}$,
the set of nonconstant solutions of 
\eqref{IS} 
form a curve represented by
\begin{equation}\label{bifexp}
(d_{1},w, \tau)=(d_{1}(s), w^{*}(d_{1}(s))+s(\varPhi_{j}+\psi (\,\cdot\,,s)),
\tau(s))
\quad\mbox{for}\ s\in [-\sigma, \sigma]
\end{equation}
with some small $\sigma >0$,
where
$d_{1}(s)$,
$\tau (s)\in\mathbb{R}_{+}$
are of $C^{1}$ class
satisfying
$d_{1}(0)=\delta_{j}$,
$\tau (0)=\tau^{*}$,
$\tau'(0)=0$,
and
$\varPhi_{j}$ is an eigenfunction satisfying
\begin{equation}\label{eg}
-\Delta\varPhi_{j}=\lambda_{j}\varPhi_{j}\quad\mbox{in}\ \Omega,
\quad \partial_{\nu}\varPhi_{j}=0\quad\mbox{on}\
\partial\Omega,\quad
\|\varPhi_{j}\|_{2}=1,
\end{equation}
and
$\psi (\,\cdot\,,s)\in W^{2,p}_{\nu}(\Omega )$ satisfies
$\psi (\,\cdot\,,0)=0$ and 
$\int_{\Omega}\psi(x,s)=\int_{\Omega }\varPhi_{j}(x)\psi (x,s)=0$
for any $|s|\le\sigma$.
\end{thm}

\begin{proof}
Suppose that $c_{1}/c_{2}<a_{1}/a_{2}<b_{1}/b_{2}$
or $b_{1}/b_{2}<a_{1}/a_{2}<c_{1}/c_{2}$.
Our aim is to construct a local curve 
$\{(w,\tau )\}$ of nonconstant
solutions of
\begin{equation}\label{ISS}
\begin{cases}
\Delta w+f(u(w,\tau ), v(w,\tau ))-\gamma
g(u(w,\tau ), v(w,\tau ))=0
\quad &\mbox{in}\ \Omega,\\
\partial_{\nu }w=0
\quad &\mbox{on}\ \partial\Omega,\\
\displaystyle\int_{\Omega}
f(u(w, \tau), v(w, \tau ))=0,
\end{cases}
\end{equation}
where 
$$u(w,\tau )=
\dfrac{\sqrt{w^{2}+4\gamma d_{1}d_{2}\tau}+w}{2d_{1}}
\quad\mbox{and}\quad
v(w, \tau )
=\dfrac{\tau}{u(w,\tau )}=
\dfrac{\sqrt{w^{2}+4\gamma d_{1}d_{2}\tau}-w}{2\gamma d_{2}}.
$$
Since $(w^{*}(d_{1}),\tau^{*})$
is a positive constant solution of \eqref{ISS}
for any $d_{1}>0$,
we shift $(w^{*}(d_{1}),\tau^{*})$ to the origin
by the change of variables
\begin{equation}\label{trans}
\phi :=w-w^{*}(d_{1})
\quad\mbox{and}\quad\xi=\tau-\tau^{*}.
\end{equation}
Hereafter we shall construct the solution curve so that
$\phi$ lies in the Banach space
$
X:=\{\phi\in W^{2,p}_{\nu }(\Omega )\,:\,
\int_{\Omega}\phi =0\}$.
To this end, we define an operator $\mathcal{F}(d_{1},\phi,\xi)\,:\,\mathbb{R}\times X\times\mathbb{R}
\to L^{p}(\Omega )\times\mathbb{R}$ associated with \eqref{IS} by
$$
\mathcal{F}(d_{1},\phi,\xi)=\biggl[
\begin{array}{c}
\mathcal{F}^{(1)}(d_{1},\phi,\xi)\\
\mathcal{F}^{(2)}(d_{1},\phi,\xi)
\end{array}
\biggr],
$$
where 
\begin{equation}\label{F1}
\begin{split}
\mathcal{F}^{(1)}(d_{1},\phi,\xi):=&
\Delta\phi+f(u(w^{*}(d_{1})+\phi, \tau^{*}+\xi),
v(w^{*}(d_{1})+\phi, \tau^{*}+\xi))\\
&
-\gamma g(u(w^{*}(d_{1})+\phi, \tau^{*}+\xi),
v(w^{*}(d_{1})+\phi, \tau^{*}+\xi))
\end{split}
\end{equation}
and
\begin{equation}\label{F2}
    \mathcal{F}^{(2)}(d_{1},\phi,\xi):=\int_{\Omega}
    f(u(w^{*}(d_{1})+\phi, \tau^{*}+\xi),
v(w^{*}(d_{1})+\phi, \tau^{*}+\xi)).
\end{equation}
In order to find bifurcation points of nonconstant solutions
of $\mathcal{F}(d_{1},\phi,\xi)=0$ on the trivial solution branch
$\{(d_{1},0,0)\,:\,d_{1}>0\}$,
we first seek for degenerate points of
the linearized operator of $\mathcal{F}$ around 
$(\phi,\xi)=(0,0)$, which will be denoted by
$$L(d_{1}):=\mathcal{F}_{(\phi,\xi)}(d_{1},0,0)\in\mathcal{L}(X\times\mathbb{R},
L^{p}(\Omega )\times\mathbb{R}),$$ that is,
\begin{equation}\label{Ldef}
L(d_{1})=
\biggl[
\begin{array}{cc}
L_{11}(d_{1}) & L_{12}(d_{1}) \\
L_{21}(d_{1}) & L_{22}(d_{1})
\end{array}
\biggr]:=\biggl[
\begin{array}{cc}
\mathcal{F}^{(1)}_{\phi}(d_{1},0,0) & \mathcal{F}^{(1)}_{\xi}(d_{1}, 0, 0)\\
\mathcal{F}^{(2)}_{\phi}(d_{1},0,0) & \mathcal{F}^{(2)}_{\xi}(d_1, 0, 0)
\end{array}
\biggr].
\end{equation}
Since 
$f(u^{*},v^{*})=g(u^{*},v^{*})=0$ and
$$
(w^{*})^{2}+4\gamma d_{1}d_{2}\tau^{*}=
(d_{1}u^{*}-\gamma d_{2}v^{*})^{2}+4\gamma d_{1}d_{2}v^{*}=
(d_{1}u^{*}+\gamma d_{2}v^{*})^{2},
$$
then a straightforward computation yields
\begin{equation}\label{jacobi}
\biggl[
\begin{array}{cc}
f_{u}^{*} & f_{v}^{*}\\
g_{u}^{*} & g_{v}^{*}
\end{array}
\biggr]
=
-
\biggl[
\begin{array}{cc}
b_{1}u^{*} & c_{1}u^{*}\\
b_{2}v^{*} & c_{2}v^{*}
\end{array}
\biggr]
\ \mbox{and}\ 
\biggl[
\begin{array}{cc}
u^{*}_{w} & u^{*}_{\tau}\\
v^{*}_{w} & v^{*}_{\tau}
\end{array}
\biggr]
=
\dfrac{1}{d_{1}u^{*}+\gamma d_{2}v^{*}}
\biggl[
\begin{array}{cc}
u^{*} & 1/(4d_{1})\\
-v^{*} & 1/(4\gamma d_{2})
\end{array}
\biggr],
\end{equation}
where $f^{*}_{u}:=f_{u}(u^{*},v^{*})$,
$u_{w}^{*}:=u_{w}(w^{*}(d_{1}),\tau^{*})$
and other notations are defined by the same manner.
It follows from \eqref{F1}-\eqref{jacobi} that
orthogonal entries of $L(d_{1})$ are computed as follows
\begin{equation}\label{L11}
\begin{split}
     L_{11}(d_{1})=&
\Delta +
f_{u}^{*}u^{*}_{w}+f_{v}^{*}v^{*}_{w}
-\gamma
(g_{u}^{*}u^{*}_{w}+g_{v}^{*}v_{w}^{*})\\
=&\Delta
+\dfrac{(c_{1}+\gamma b_{2})\tau^{*}-b_{1}(u^{*})^{2}-\gamma c_{2}(v^{*})^{2}}
{d_{1}u^{*}+\gamma d_{2}v^{*}}.
\end{split}
\end{equation}
and
\begin{equation}\label{L22}
    L_{22}(d_{1})=(f^{*}_{u}u^{*}_{\tau}+f^{*}_{v}v^{*}_{\tau })|\Omega |
    =-\dfrac{u^{*}|\Omega|}{4(d_{1}u^{*}+\gamma d_{2}v^{*})}
    \biggl(\dfrac{b_{1}}{d_{1}}+\dfrac{c_{1}}{\gamma d_{2}}\biggr).
\end{equation}
Here we remark that
$$
L_{21}(d_{1})\phi=
(f^{*}_{u}u_{w}^{*}+f^{*}_{v}v^{*}_{w})\int_{\Omega}\phi=0
\quad\mbox{for any}\ \phi\in X,
$$
that is to say,
$L_{21}(d_{1})=0$ for any $d_{1}>0$.
Since \eqref{L22} implies that $L_{22}(d_{1})<0$ for any $d_{1}>0$,
we have only to investigate the degeneracy of 
$L_{11}(d_{1})\in\mathcal{L}(X,L^{p}(\Omega ))$.
In view of the potential term of $L_{11}(d_{1})$ in \eqref{L11},
we use \eqref{const} and \eqref{taus} to see that
$$
\lim_{b_{1}\to +0,\,
c_{2}\to +0}
\dfrac{(c_{1}+\gamma b_{2})\tau^{*}-b_{1}(u^{*})^{2}-\gamma c_{2}(v^{*})^{2}}
{d_{1}u^{*}+\gamma d_{2}v^{*}}
=
\dfrac{(\gamma b_{2}+c_{1})a_{1}a_{2}}
{d_{1}a_{2}c_{1}+\gamma d_{2}a_{1}b_{2}}>0.
$$
Then,
for any positive eigenvalue $\lambda_{j}$
whose eigenspace is one dimension,
there exists a small positive number $\eta_{j}$
such that if
$0<b_{1},\,c_{2},\,d_{2}<\eta_{j}$, then
\begin{equation}\label{decpass}
\dfrac{(c_{1}+\gamma b_{2})\tau^{*}-b_{1}(u^{*})^{2}-\gamma c_{2}(v^{*})^{2}}
{d_{1}u^{*}+\gamma d_{2}v^{*}}
\begin{cases}
>\lambda_{j}\quad &\mbox{for}\ 0<d_{1}<\delta_{j},\\
=\lambda_{j}\quad &\mbox{for}\ d_{1}=\delta_{j},\\
<\lambda_{j}\quad &\mbox{for}\ d_{1}>\delta_{j}
\end{cases}
\end{equation}
with some $\delta_{j}>0$.
Since $L_{11}(\delta_{j})=\Delta +\lambda_{j}$,
then $\mbox{Ker}\,L_{11}(\delta_{j})=\mbox{Span}\{\varPhi_{j}\}\subset X$,
where $\varPhi_{j}$ is an eigenfunction satisfying \eqref{eg}.
It is noted that $\int_{\Omega}\varPhi_{j}=0$.
Together with $L_{21}(\delta_{j})=0$ and $L_{22}(\delta_{j})<0$,
we know that
$
\mbox{Ker}\,L(\delta_{j})=\{\,(\phi,\xi )=t(\varPhi_{j},0)\,:\,
t\in\mathbb{R}\}$.
In order to use the local bifurcation theorem \cite[Theorem 1.7]{CR},
we have to check the following transversality condition
\begin{equation}\label{trans2}
\mathcal{F}_{(\phi,\xi),d_{1}}(\delta_{j},0,0)
\biggl[\begin{array}{c}
\varPhi_{j}\\
0
\end{array}
\biggr]\not\in
\mbox{Ran}\,L(\delta_{j}).
\end{equation}
To this end, it obviously suffices to show
$ \mathcal{F}^{(1)}_{\phi,d_{1}}(\delta_{j},0,0)\varPhi_{j}\not\in \mbox{Ran}\,L_{11}(\delta_{j})$.
Suppose for contradiction that
$ \mathcal{F}^{(1)}_{\phi,d_{1}}(\delta_{j},0,0)\varPhi_{j}\in \mbox{Ran}\,L_{11}(\delta_{j})$.
It is possible to verify that
$$
\mathcal{F}^{(1)}_{\phi, d_{1}}(\delta_{j},0,0)\,\varPhi_{j}
=-
\dfrac{u^{*}\{(c_{1}+\gamma b_{2})\tau^{*}-b_{1}(u^{*})^{2}-\gamma c_{2}(v^{*})^{2}\}}{(\delta_{j}u^{*}+\gamma d_{2}v^{*})^{2}}
\varPhi_{j}
=
-\dfrac{u^{*}\lambda_{j}}{\delta_{j}u^{*}+\gamma d_{2}v^{*}}\varPhi_{j},
$$
where the last equality comes from \eqref{decpass}. 
By virtue of
the Fredholm alternative theorem,
one can see 
$
\mbox{Ran}\,L_{11}(\delta_{j})=\{\psi\in L^{p}(\Omega )\,:\,
\int_{\Omega}\psi=\int_{\Omega}\psi\varPhi_{j}\,dx=0\,\}$.
Then our assumption is equivalent to
$$
-
\dfrac{u^{*}\lambda_{j}}{\delta_{j}u^{*}+\gamma d_{2}v^{*}}
\|\varPhi_{j}\|^{2}_{2}=0.
$$
This obviously contradicts $\|\varPhi_{j}\|_{2}=1$.
Therefore, the transversality condition \eqref{trans2} holds true.
Consequently, we have checked all conditions for use of 
the local bifurcation theorem \cite[Theorem 1.7]{CR} to 
obtain the bifurcation curve of nonconstant solutions
expressed as \eqref{bifexp} by way of \eqref{trans}.
We complete the proof of Theorem \ref{bifthm}.
\end{proof}

\begin{rem}
By the bifurcation theorem by Krasnoselski \cite{Kr},
we can show that $(d_{1},w, \tau)=(\delta_{j}, w^{*}, \tau^{*})$ is still
a bifurcation point in some sense under a weaker assumption on
$\lambda_{j}$ that its multiplicity is odd.
\end{rem}

Concerning the other limiting system \eqref{CS} of complete segregation,
since it is also a fast reaction limiting system of the Lotka-Volterra
competition model with linear diffusion terms
(see e.g., \cite{DD, DHMP, DZ, HY, Ka2}).
In particular,
for the one-dimensional case,
Dancer, Hilhorst, Mimura and Peletier \cite[Theorem 4.1]{DHMP}
obtained the following detailed 
structure of nontrivial solutions of \eqref{CS}.
\begin{thm}[\cite{DHMP}]
Suppose that $\Omega =(0,1)$.
If $\sqrt{d_{1}/a_{1}}+\sqrt{d_{2}/a_{2}}\ge 2/\pi$,
there is no nonconstant solution of \eqref{CS}.
For each $n\in\mathbb{N}$,
if $\sqrt{d_{1}/a_{1}}+\sqrt{d_{2}/a_{2}}< 2/(n\pi)$, then
\eqref{CS} has $n$-time(s) sign-changing solutions
$w^{(n)}_{fg}$,
$w^{(n)}_{gf}\in C^{2}(\overline{\Omega })$ in the sense that
the number
of zeros of each is $n$.
More precisely,
$w^{(n)}_{fg}$ is expressed as
$$
w^{(n)}_{fg}(x)=\begin{cases}
\sum\limits^{[n/2]}_{j=1}\psi (x-2j/n) + \phi (x-(n-1)/n)\quad &\mbox{if $n$ is odd},\\
\sum\limits^{[n/2]}_{j=1}\psi (x-2j/n) \quad &\mbox{if $n$ is even}
\end{cases}
$$
for
$x\in\Omega$. Here,
$\phi $ is a unit part determining the profile of $w^{(n)}_{fg}$
defined by
$$
\phi(x)=\begin{cases}
d_{1}u(x)>0\quad&\mbox{if}\ x\in [0,\theta_{n}),\\
-\gamma d_{2}v(x)<0\quad&\mbox{if}\ x\in (\theta_{n},1/n],\\
0\quad &\mbox{otherwise}
\end{cases}
$$
with some $\theta_{n}\in (0,1/n)$,
and $\psi (x)=\phi (x)+\phi(2/n-x)$,
where $u(x)$ 
$(x\in [0,\theta_{n}])$
and $v(x)$ 
$(x\in [\theta_{n},1/n])$ are solutions of
$$
\begin{cases}
d_{1}u''+u(a_{1}-b_{1}u)=0,\quad u>0>u'\quad&\mbox{in}\ (0,\theta_{n}),\\
d_{2}v''+v(a_{2}-c_{2}v)=0,\quad v,\,v'>0\quad&\mbox{in}\ (\theta_{n},1/n),\\
u(\theta_{n})=v(\theta_{n})=0,\quad
d_{1}u'(\theta_{n})=-\gamma d_{2}v'(\theta_{n}),\\
u'(0)=v'(1/n)=0.
\end{cases}
$$
On the other hand, $w^{(n)}_{gf}$ is expressed by
$$
w^{(n)}_{gf}(x)=\begin{cases}
w^{(n)}_{fg}(1-x)\quad &\mbox{if $n$ is odd},\\
w^{(n)}_{fg}(x+1/n)+\phi(x-(n-1)n)\ &\mbox{if $n$ is even}.
\end{cases}
$$

\end{thm}

%
%

\end{document}